\documentclass[11pt,reqno]{amsart}

\usepackage{graphicx}
\usepackage{color}
\usepackage{geometry}
\usepackage[all]{xy}
\usepackage{tikz}
\usetikzlibrary{decorations.pathreplacing}
\usetikzlibrary{arrows}

\usepackage{calrsfs}
\usepackage{bbold}
\usepackage[T1]{fontenc} 
\usepackage{textcomp}
\usepackage{times}
\usepackage[scaled=0.92]{helvet} 

\font\eulersm=eusm10 at 11pt
\def\esm#1{\hbox{\eulersm {#1}}}
\newcommand{\T}{\esm{T}}

\font\eulersmm=eusm10 at 9pt
\def\esmm#1{\hbox{\eulersmm {#1}}}
\newcommand{\sT}{\esmm{T}}

\newtheorem{theorem}{Theorem}[section]

\newtheorem{lemma}[theorem]{Lemma}

\theoremstyle{definition}
\newtheorem{definition}[theorem]{Definition}

\usepackage{amsfonts}
\usepackage{amssymb}

\makeatletter
\newcommand{\oset}[3][0ex]{%
  \mathrel{\mathop{#3}\limits^{
    \vbox to#1{\kern-\ex@
    \hbox{$\scriptstyle#2$}\vss}}}}
\makeatother
\def\la#1{\oset{\vspace*{-5mm}\leftarrow}{#1}}
\def\ra#1{\oset{\rightarrow}{#1}}

\renewcommand{\bar}{\overline}
\renewcommand{\hat}{\widehat}

\DeclareMathOperator{\N}{\mathbf{N}}

\DeclareMathOperator{\R}{\mathbf{R}}

\DeclareMathOperator{\E}{\mathbf{E}}
\DeclareMathOperator{\F}{\mathbb{F}}
\DeclareMathOperator{\TT}{\mathbb{T}}
\DeclareMathOperator{\PF}{\mathrm{PF}}

\DeclareMathOperator{\Cyl}{Cyl}

\definecolor{green}{RGB}{0,170,100}

\begin{document}

\date{\today\ (version 0.1)} 
\title{Rigidity and reconstruction for graphs}
\author[G.~Cornelissen]{Gunther Cornelissen}
\address{\normalfont Mathematisch Instituut, Universiteit Utrecht, Postbus 80.010, NL-3508 TA Utrecht}
\email{g.cornelissen@uu.nl}
\author[J.~Kool]{Janne Kool}
\address{\normalfont  Institut for Matematiske Fag, K{\o}benhavns Universitet, Universitetsparken 5, 2100 K{\o}benhavn \O}
\email{{j.kool@math.ku.dk}}
\thanks{Part of this work was done while 
the second author visited the Max-Planck-Institute in Bonn. We thank Tom Kempton and Matilde Marcolli for their input.}


\begin{abstract} \noindent We present measure theoretic rigidity for graphs of first Betti number $b>1$ in terms of measures on the boundary of a $2b$-regular tree, that we make explicit in terms of the edge-adjacency and closed-walk structure of the graph. We prove that edge-reconstruction of the entire graph is equivalent to that of the ``closed walk lengths''. 
\end{abstract}

\maketitle


%


\section{Some rigidity phenomena} 
A compact Riemann surface $X$ of genus $g \geq 2$ is uniquely determined by a dynamical system, namely, the action of the fundamental group $\Pi_g$ in genus $g$ on the Poincar\'e disk $\Delta$ by M\"obius transformations. Things change when we replace $\Delta$ by its real one-dimensional boundary $\partial \Delta = S^1$; the action of $\Pi_g$ extends to $S^1$, but this action will only depend on the \emph{topological} isomorphism type of $X$, viz., the genus $g$. Rigidity re-enters the picture via the Lebesgue-measure on $S^1$, in the sense that two Riemann surfaces $X$ and $Y$ are isomorphic if and only if there exists a $\Pi_g$-equivariant absolutely continuous homeomorphism $S^1 \rightarrow S^1$ (cf.\ e.g.\ \cite{Kuusalo}). 

A similar result holds for more general hyperbolic spaces. We describe a version for graphs (cf.\ Coornaert \cite{CoornaertCRAS}): 
let $G=(V,E)$ denote a graph with vertex set $V$ and edge set $E$, consisting of unordered pairs of elements of $V$. Let $b$ denote the first Betti number of $G$, and assume $b \geq 2$. Knowing $G$ is the same as knowing the action of a free group $\F_b$ or rank $b$ on the universal covering tree $\T$ of $G$. Again, the dynamical system of $\F_b$ acting on the boundary $\partial \T$ of $\T$ (i.e., the space of ends of $\T$) is topologically conjugate to a system that only depends on $b$ (to wit, the action of $\F_b$ on the boundary of its Cayley graph), but if one considers the set of Patterson-Sullivan measures on $\partial \T$, rigidity holds; we provide an exact result in Theorem \ref{rig} below. 

The graph rigidity theorem shows the importance of the structure of the ``space of closed walks'' $(C,\F_b,\mu)$ of a graph in understanding the structure of a graph. We apply this insight to \emph{reconstruction problems} for graphs. 
We find that for average degree $>4$, we can reconstruct various ingredients of the explicit formula for the rigidifying measure. We conclude that reconstruction of a graph is intimately related with the structure of lengths of closed walks in a graph. The final section confirms this; we prove in an elementary way that the edge reconstruction conjecture is equivalent to the reconstruction of ``closed walks and their length''. 

One may extend rigidity from graphs to curves over non-archimedean fields. In \cite{CorKool}, we have explained how, for such a generalization, one needs the require boundary homomorphisms to respect relations between so-called \emph{harmonic measures}. 
It would be interesting to use such insights to formulate reconstruction problems for such curves, and for Riemann surfaces.

\section{Dynamics on the boundary of the universal covering tree} 
In this section, we formulate a precise rigidity theorem for graphs, in terms of measures induced by Patterson-Sullivan measure on a ``universal'' topological dynamical system (``universal'' in that it only depends on the first Betti number of the graph).

Let $G$ denote a graph with vertex set $V$ and edge set $E$. Assume that the first Betti number $b$ of $G$ satisfies $b=b_1(G)>1$ and that $G$ does not have ends, i.e., vertices of degree 1. Let $\T$ denote the universal covering tree of $G$, so that $G$ is the quotient of $\T$ by its fundamental group $\Gamma \cong \F_{b}$, a free group of rank  $b$. By assumption, $\T$ has no end-vertices and is locally compact. Let $\partial \T$ denote the (topological) space of ends of $\T$, on which $\Gamma$ acts. The space $\partial \T$ consists of equivalence classes of half lines in $\T$, where two half lines are equivalent if they differ in only finitely many edges. Fixing a base point $x_0 \in \T$, the space $\partial \T$ consists of all half lines $ p \colon \N \rightarrow V(\T)$ with $p(0)=x_0$, and for all $n \geq 1$, $p(n) \neq p(n-1)$ and $p(n-1) \neq p(n+1)$.
The Borel sigma-algebra  of $\partial \T$ is spanned by a basis of clopen sets, the \emph{cylinder sets} $\mathrm{Cyl}_{x_0}(f)$, where $f$ runs through the edges $f \in E(\T)$. Here, a cylinder set $\mathrm{Cyl}_{x_0}(f)$ consists of classes of half-lines that originate from $x_0$ and pass through $f$.

The above definition is that of the so-called \emph{visual boundary} of $\T$. It is also possible to define $\partial \T $ as the \emph{hyperbolic boundary} $\partial \T = \bar{\T}-\T$, where $\bar \T$ is a metric completion of $\T$ in a suitable hyperbolic metric on $\T$. The two notions coincide in our case, so we will use them interchangeably; cf.\ Chapter 2 in \cite{CDP}. 

Let $\mathrm{Cay}(\F_{b})$ denote the Cayley graph of $\F_{b}$, for any chosen symmetrization of a set of generators $g_1,\dots,g_{b}$. This is a $2b$-regular tree. Let $C: = \partial \mathrm{Cay}(\F_{b})$ denote its boundary. The Borel sigma-algebra of $C = \partial \mathrm{Cay}(\F_b)$ is spanned by the cylinder sets $\mathrm{Cyl}(g):=\mathrm{Cyl}_1(g)$ for $g \in \F_b-\{1\}$, given as the set of limits of reduced words that begin with $g$.

\begin{lemma}\label{topconj} The tree $\T$ is quasi-isometric to $\mathrm{Cay}(\F_{b})$, and there is a topological conjugacy 
$(\Phi_G,\alpha_G) \colon (\partial\T,\Gamma) \rightarrow (C, \F_{b})$ 
of dynamical systems, i.e., \begin{enumerate} \item  $\Phi_G \colon \partial \T \rightarrow C$ is a homeomorphism; \item $\alpha_G \colon \Gamma \rightarrow \F_{b}$ is a group isomorphism; \item The equivariance $\Phi_G(\gamma x) = \alpha_G(\gamma) \Phi_G(x)$ holds for all $x \in \partial \T$ and $\gamma \in \Gamma$. \end{enumerate} 
\end{lemma} 

\begin{proof} (See, e.g., \cite{CDP}, Theorem 4.1.)  Choose a base point $x_0 \in \T$ and an isomorphism $\alpha \colon \F_b \rightarrow \Gamma$. Now $$ \varphi \colon \mathrm{Cay}(\F_b) \rightarrow \T \colon g \mapsto \alpha(g)(x_0) $$ is a quasi-isometry and extends to a boundary homeomorphism $$\Phi \colon C \rightarrow \partial \T \colon \lim g \mapsto \lim \alpha(g)(x_0)$$ (loc.\ cit., 2.2).  The resulting limit map is obviously equivariant w.r.t.\ the group isomorphism $\alpha$, since for $h \in \F_b$ and $\lim g \in C$, we have $$\Phi(h \lim g) = \Phi(\lim hg) = \lim \alpha(h)\alpha(g)(x)= \alpha(h) \lim \alpha(g)(x_0) = \alpha(h) \Phi(\lim g).$$ Hence we can set $\alpha_G=\alpha^{-1}$ and $\Phi_G = \Phi^{-1}$. 
\end{proof}

\begin{definition} \label{busi} Let $d(\cdot,\cdot)$ denote the distance between the vertices of $\T$ (i.e., such that $d(v,w)=1$ if $v$ and $w$ are adjacent). 
For $\xi\in\partial\T$ and $x,y\in \T$ the \emph{Busemann function} is defined by
\[B_{\xi}(x,y)=\lim_{\substack{z\in{\small \sT}\\z\to\xi}}(d(x,z)-d(y,z)).\]

A family of positive finite Borel measures $\{\mu_x\}_{x \in \sT}$ on $\partial \T$ is called $\Gamma$-\emph{conformal of dimension $\delta$} if it satisfies the following properties:
\begin{enumerate}
 \item the family $(\mu_x)_x$ is $\Gamma$-equivariant, i.e., $\mu_{\gamma x}=(\gamma^{-1})_{\ast}\mu_x$, $\forall x\in\T, \gamma\in\Gamma$.
 \item for all $x,y\in\T$ the Radon-Nikodym derivative of $\mu_x$ with respect to $\mu_y$ exists and equals $d\mu_x/d\mu_y(\xi)=e^{-\delta B_{\xi}(x,y)}.$ 
\end{enumerate}
Observe that if $\xi \in \Cyl_x(f)$, then \begin{equation} \label{cst} B_{\xi}(x,o(f)) = d(x,\xi)-d(o(f),\xi) = d(x,o(f))\end{equation} is constant in $\xi$.
We consider such measures only up to scaling by a global constant. 
\end{definition}

Families of $\Gamma$-conformal measures exist: let $\mu_x=\mu_{G,x}$ denote the family of Patterson-Sullivan measures for the action of $\Gamma$ on $\T$, based at some point  $x \in \T$, defined as the weak limit of measures 
$$ \mu_{x} = \lim_{ s \rightarrow \log(\lambda)} \frac{\sum\limits_{\gamma \in \Gamma} e^{-s d(x,\gamma x)} \delta_{\gamma x}}{\sum\limits_{\gamma \in \Gamma} e^{-sd(x,\gamma x)}}, $$
for suitable (unique) $\lambda$ (actually, $\lambda=\lambda_{\PF}$ is the Perron-Frobenius eigenvalue of the edge-adjacency operator $T$ defined in the next section), and $\delta_x$ is the Dirac delta measure at $x$; the dimension of the measures $\mu_x$ is $ \log \lambda$.

Recall that two measures $\mu_1$ and $\mu_2$ on a space $Y$ are called \emph{mutually absolutely continuous} if $\mu_1(A)=0 \iff \mu_2(A)=0$ for all measurable $A \subseteq Y$. A measurable map $\varphi \colon (X,\mu_X) \rightarrow (Y,\mu_Y)$ between measure spaces is absolutely continuous if $\Phi_* \mu_X$ and $\mu_Y$ are mutually absolutely continuous on $Y$. Here, the \emph{push-forward measure} $\varphi_* \mu_X$ is defined by $\varphi_*\mu_X(A):= \mu_X(\varphi^{-1}(A))$ for measurable $A \subseteq Y$. 
We have the following measure-theoretic rigidity theorem: 

\begin{theorem} \label{rig} Let $G$ denote a graph of minimal degree $\geq 3$ and first Betti number $b>1$, with universal covering tree $\T$ and fundamental group $\Gamma$, and let $\mu$ denote a $\Gamma$-conformal measure on $\T$. Let $(\T',\Gamma',\mu')$ denote the same data associated to another graph $G'$ of minimal degree $\geq 3$ with the same first Betti number $b>1$. Then $G$ and $G'$ are isomorphic if and only if the push-forward measures $\Phi_{G\ast} \mu$ and $\Phi_{G'\ast} \mu'$ on $C$ have the same dimension and are mutually absolutely continuous. 
\end{theorem}

\begin{proof} If $G$ and $G'$ are isomorphic, there is nothing to prove. For the converse direction, 
from \cite{CorKool}, Theorem 2.7, we recall measure-theoretic rigidity for graphs (the case of graphs with the \emph{same} covering trees was proven by Coornaert in \cite{CoornaertCRAS}): the graphs $G$ (corresponding to $(\partial \T, \Gamma, \mu)$) and $G'$  (corresponding to $(\partial \T', \Gamma', \mu')$) are isomorphic if and only if there exists
a group isomorphism $\alpha \colon \Gamma \rightarrow \Gamma'$
and 
a homeomorphism $\varphi \colon \partial \T \rightarrow \partial \T'$
such that 
\begin{enumerate}
\item[\textup{(a)}] $\varphi$ is $\alpha$-equivariant, i.e., we have $\varphi(\gamma x) = \alpha(\gamma) \varphi(x)$, $\forall x \in \partial \T, \gamma \in \Gamma$;
\item[\textup{(b)}] the measures $\mu$ and $\mu'$ have the same dimension; 
\item[\textup{(c)}]$\varphi$ is absolutely continuous w.r.t.\ $\mu$ and $\mu'$.
\end{enumerate}

The theorem is a reformulation of this result by passing to the fixed space $(C,\F_b)$: set $\varphi:=\Phi_{G'}^{-1} \circ \Phi_{G} \mbox{ and } \alpha:=\alpha_{G'}^{-1} \circ \alpha_{G}.$ If $\Phi_{G\ast} \mu$ and $\Phi_{G'\ast} \mu'$ are absolutely continuous and have the same dimension, then $(\varphi,\alpha)$ satisfy the five listed conditions (1)-(2) and (a)-(c), so $G$ and $G'$ are isomorphic. 
\end{proof}

\section{The measure and the $T$-operator} 

In this section, we show that the Patterson-Sullivan measure can be described in terms of the Perron-Frobenius eigenspace of a certain ``edge adjacency operator''. 

If $e=\{v_1,v_2\} \in E$, we denote by $\ra{e}\ =(v_1,v_2)$ the edge $e$ with a chosen orientation, and by $\la{e}\ =(v_2,v_1)$ the same edge with the inverse orientation to that of $\ra{e}$. Let $o(\ra{e})=v_1$ denote the origin of $\ra{e}$ and $t(\ra{e})=v_2$ its end point. 
The \emph{edge adjacency matrix} $T=T_G$ (compare \cite{Terras}) is defined as follows; let $\E$ denote the set of oriented edges of $G$ for any possible choice of orientation, so $|\E|=2|E|$, $T$ is defined to be the $2|E|\times2|E|$ matrix, in which the rows and columns are indexed by $\E$, and 
$$T_{\ra{e_1},\ra{e_2}}= \left\{ \begin{array}{l} 1 \mbox{ if }t(\ra{e_1})=o(\ra{e_2}) \mbox{ but }\ra{e_2}\ \neq\ \la{e_1}; \\ 0 \mbox{ otherwise}. \end{array} \right.$$ Since $b \geq 2$,  T is an irreducible non-negative matrix.  

Fix a base point $x_0 \in \T$ and let $v_0$ denote the corresponding vertex in the graph $G$. Choose a spanning tree $\mathcal B$ for $G$, and let $\{e_1,\dots,e_b\}$ denote the set of edges outside $\mathcal B$; choose an orientation on $e_i$. Let $\{\gamma_1,\dots,\gamma_b\}$ denote a set of generators for the fundamental group of $G$, seen as closed walks based at $v_0$ through an isomorphism $\Gamma \rightarrow \pi_1(G,v_0)$, such that $\gamma_i$ is a closed walk that passes through $\mathcal B$ and $\gamma_i$, but not through $\gamma_j$ for $j \neq i$.  Choose an isomorphism $\alpha \colon \F_b \rightarrow  \Gamma$ and set $g_i:=\alpha^{-1}(\gamma_i)$ as generators for $\F_b$. 
For $\gamma \in \Gamma$, let $\ell(\gamma)$ denote the \emph{length} of the closed walk $\gamma$. 
If $g \in \F_b$, we define the \emph{final edge} $\tau(g) \in \E$ of $g$ as follows: $$\tau(g)=\left\{ \begin{array}{l} \ra{e_i} \mbox{ if $g_i$ is the final letter of $g$}; \\ \la{e_i} \mbox{ if $g_i^{-1}$ is the final letter of $g$}; \end{array} \right.$$ where $g$ is written as a reduced word in the alphabet $\{g_i\}$. 
We also define the \emph{lifted final edge} $\hat \tau (g) \in \T$ of $g$ to be the last occurence of a lift of $\tau(g)$ from $G$ to $\T$ on the directed path from $x_0$ to $\alpha(g)(x_0)$ in $\T$. 

\begin{lemma}
For any $g \in \F_b-\{1\}$, we have 
$\Phi_G^{-1}(\Cyl(g))=\Cyl_{x_0}(\hat \tau(g)).$
\end{lemma}

\begin{proof}
Let $g \in \F_b-\{1\}$ and let $g_i$ denote the final letter of $g$. Since $\Cyl(g)$ consists of half infinite words starting $g g_{j} \dots$ with $g_j \neq g_i^{-1}$, under $\Phi_G^{-1}$, this is mapped to half infinite sequences of edges in $\T$, in which there can some backtracking (if there is cancellation of edges between $g_i$ and $g_j$), but never beyond $\hat \tau(g)$. Since $\partial \T$ is the universal cover of $G$, backtracking occurs precisely up to $\hat\tau(g)$, and $\Phi_G^{-1}(\Cyl(g))$ is the cylinder set of $\hat \tau(g)$. \end{proof}
 We now give an intrinsic formula for the push-forward of Patterson-Sulllivan measure to the boundary $C$ of the Cayley graph, purely in terms of data related to the original graph: \begin{theorem} \label{form} The push-forward measure $\mu=\Phi_{G\ast} \mu_{x_0}$ on $C$ is characterised (up to scaling by a global constant) by 
\begin{equation} \label{psm} \mu (\mathrm{Cyl}(g)) = \lambda_{\PF}^{-\ell(\alpha(g)) + d_{G-\tau(g)} (v_0,o(\tau(g)))} \cdot \mathbf{p}_{\tau(g)}, \end{equation}
where $\lambda_{\PF}$ is the Perron-Frobenius eigenvalue, and $\mathbf{p}$ a Perron-Frobenius eigenvector for $T$. 
\end{theorem}   

\begin{proof} We argue as in \cite{Kapovich}, 3.13, 4.2, 4.3. The conformal dimension of the Patterson-Sullilvan measures is $\log \lambda_{\PF}$. Suppose that $\ra{e}$ runs through $\E$, and $\ra{e}'$ runs through a set of lifts of $\ra{e}$ to $\T$, where $x_e=o(\ra{e}')$ is the origin of the lift $\ra{e}'$.  Define a vector $w \in \R^{2|E|}$ by $w_{\ra{e}}:=\mu_{x_e}(\mathrm{Cyl}_{x_e}(\ra{e}')).$ Then $w$ satisfies the equation $ T w = \lambda_{\PF} \cdot w,$ by the conformality of the measure and using (\ref{cst}) to move from adjacent edges back to the original edge. Since $w$ is non-negative, it is unique (up to global scaling; by Perron-Frobenius theory), so, up to scaling, equal to $\mathbf{p}$. 
By the previous lemma, we have 
$ \mu (\mathrm{Cyl}(g)) =  \mu_{x_0}(\mathrm{Cyl}_{x_0}(\hat \tau(g)). $
The conformality property (2) from Definition \ref{busi} implies that 
\begin{equation*} \mu_{x_0}(\mathrm{Cyl}_{x_0}(\hat \tau(g))=\lambda_{\PF}^{-d(x_0,o(\hat \tau(g)))} \cdot \mu_{o(\hat \tau(g))}(\mathrm{Cyl}(\hat \tau(g))). \end{equation*}
Now we have just seen that $\mu_{o(\hat \tau(g))}(\mathrm{Cyl}(\hat \tau(g))) = \mathbf{p}_{\tau(g)}$, and again, conformality and (\ref{cst}) imply that \begin{equation*} d(x_0,o(\hat \tau(g))) = \ell(\alpha(g)) - d_{G-\tau(g)} (v_0,o(\tau(g))). \qedhere \end{equation*}  
\end{proof} 


\section{Reconstruction of measure-theoretic invariants} 


The \emph{edge deck} $\mathcal{D}^e(G)$ of $G$ is the multi-set of isomorphism classes of all edge-deleted subgraphs of $G$. Harary 
conjectured in 1964 that graphs on at least four edges are edge-reconstructible, i.e., determined up to isomorphism by their edge deck. This is the \emph{edge reconstruction conjecture} (ERC), the analogue for edges of the famous vertex reconstruction conjecture (VRC) of Kelly and Ulam that every graph on at least three vertices is determined by its vertex deck (compare \cite{Bondy}). 
The degree of a vertex is the number of edges to which it belongs. 
The \emph{average degree} $\bar d$ of $G$ then equals $\bar d = \sum_{v \in V} \deg v/|V| = 2 {|E|}/{|V|}.$ In \cite{CK2}, we have proven the following: 

\begin{theorem} \label{main} Let $G$ denote a graph of average degree $\bar d>4$. Then $\lambda_{\PF}$ and
the function that associates to an element $G-e$ of the edge deck the unordered pair $\{\mathbf{p}_{\ra{e}},\mathbf{p}_{\la{e}}\}$ are edge-reconstructible. \qed
\end{theorem}

Looking at Theorems \ref{form} and \ref{main} simultaneously indicates that, from this point of view, reconstruction is intimately related to knowledge about lengths of closed walks.

\section{Reconstruction of closed walks lengths and the ERC} 

We give a direct and elementary proof of the result alluded to in the previous section, that knowing the ``structure of lengths on the space of closed walks'' determines the graph uniquely, formulated precisely in Theorem \ref{ERC-cond} below in terms of overlap lengths of loops in special spanning trees of the graph. 

Let $G$ denote a connected graph with minimal degree $\delta \geq 2$, without cut vertices in which every element of the edge deck is connected and has at most one vertex of degree $\delta-1$.  Fix an element $H=G-e$ of the edge deck of $G$ that contains at least one vertex $\alpha$ of degree $\delta-1$. The missing edge connects $\alpha$ to another (to-be-reconstructed) vertex $\omega \in H$. 

The assumptions are not restrictive in view of the ERC. The edge deck determines the vertex deck (
compare \cite{Bondy}, 6.13), so if VRC is known for some graph, then ERC also holds for such graphs. Since VRC (hence ERC) is known for disconnected graphs (
compare \cite{Bondy} 4.6), we can assume that $G$ is connected. Also, since VRC (hence ERC) is known for graphs with a cut vertex without pendant vertices (Bondy, \cite{BondyPacific}) and we assume that $G$ has no pendant vertices (since $\delta \geq 2$), we can assume that all cards in the edge deck are connected: if such a card is disconnected, any of the end points of the missing edge would be a cut vertex of $G$. Hence any card $H=G-e$ has first Betti number $b_1(H)=b-1$.  
Finally, if $H$ contains another vertex of degree $\delta-1$, then this vertex is $\omega$, and the problem is solved. Hence we can assume that all vertices except $\alpha$ have degree at least $\delta$ in $H$. 

Orient the missing edge $\ra{e}$ such that it has origin $o(\ra{e})=\alpha$. 
Now let $\gamma_0$ denote an embedded closed walk from $\alpha$ to $\alpha$ through $\ra{e}$ of minimal length $\ell_0$, that passes through $e$ exactly once.  Such a closed walk exists: since $H$ is connected, there exists a shortest path $P$ in $H$ from $\omega$ to $\alpha$, which we can close by adding the edge $e$.

\begin{lemma} The length \label{ell0} $\ell_0$ is edge-reconstructible.
\end{lemma}

\begin{proof} 
Indeed, $\ell_0 = \min \{ r  \colon S_r(G)-S_r(G-e)>0\},$ where $S_r(G)$ is the number of subgraphs of a graph $G$ isomorphic to the cycle graph $C_r$.

Now $\ell_0 < |E|$.  Indeed, the length of the path $P$ is at most $|E(H)|$. In case $P$ has length $|E(H)|$, then $H$ is itself a path, and $G$ is a cycle graph. But a cycle graph with at least four edges has two adjacent vertices of minimal degree, which we assume is not the case. 
Now for $r < |E|$, $S_r(G)$ is edge-reconstructible by Kelly's Lemma (\cite{Bondy}, 6.6).
\end{proof} 

In the minimal case where $\delta = 2$, i.e., if $\deg_H \alpha = 1$, we make some replacements: we denote by $H$ the graph $H-\alpha$; we denote by $\alpha$ the vertex of the new $H$ that corresponds to the (unique) vertex adjacent to the original $\alpha$ in the original $H$, and we replace $\ell_0$ by $\ell_0-1$. After these replacements, we can assume that $H$ had minimal degree $\geq 2$. 

\begin{lemma} \label{tree}
If a vertex $v\in V$ is not equal or adjacent to $\alpha$ in $H$ and has degree $\deg_Hv\geq 3$, then there exist oriented edges $\ra{e_1}$ and $\ra{e_v}$ in $H$, such that $t(\ra{e_1})=\alpha$, $t(\ra{e_v})=v$, and such that there exists a spanning tree $\TT$ for $H$ with $e_1\not\in\TT$ and $e_v\not\in\TT$. If $\tilde v$ is a vertex adjacent to $v$, we can furthermore guarantee that $o(\ra{e_v})\not=\tilde v$.

If $v\in V$ is not equal but adjacent to $\alpha$ and $\deg_H\alpha\geq3$, then there exist oriented edges $\ra{e_1}$ and $\ra{e_v}$ in $H$, such that $t(\ra{e_1})=\alpha$, $t(\ra{e_v})=v$,  $o(\ra{e_v})\not=\alpha$, and such that there exists a spanning tree $\TT$ for $H$ with $e_1\not\in\TT$ and $e_v\not\in\TT$.
\end{lemma}

\begin{proof}
It suffices to prove that there are edges $e_1$ incident to $\alpha$, $e_v$ incident to $v$, $e_1\not=e_v$ such that $H-e_1-e_v$ is connected, since then we can set $\TT$ to be the spanning tree of $H-e_1-e_v$. Note that the number of connected components of $H-v$ is at most $2$; the missing edge $e$ connects at most two components and the graph $G$ has no cut-vertex by assumption. Since $\deg v\geq 3$ there is at least one connected component $C$ of $H-v$ such that there are two edges connecting $v$ and $C$.

First, assume that $\alpha$ is not adjacent to $v$. If $\alpha\not\in C$, but in the other component $C'$, then let $e_v$ be one of the edges connecting $v$ and $C$, and let $e_1$ be any edge incident to $\alpha$ such that $H-e_1$ is connected. Note that this is possible since if there is an edge $\tilde e$ incident to $\alpha$ contained in $C'$ such that $C'-\tilde e$ is not connected, $\alpha$ would have been a cut-vertex in $G$.  
If $\alpha\in C$ and $C-\alpha$ is connected, then let $e_v$ be one of the edges connecting $v$ and $C$ and set $e_1$ to be any edge incident with $\alpha$. If $C-\alpha$ is not connected, then any of the connected components of $C-\alpha$ is connected to $v$, since otherwise $\alpha$ would have been a cut-vertex. Pick $e_v$ an $e_1$ such that they connect $v$ and $\alpha$ with different connected components of $C-\alpha$ respectively. It is clear that in both cases the choice can be made such that $o(\ra{e_v})\not=\tilde v$ for any vertex $\tilde v\not=\alpha$ adjacent to $v$.

Next, assume that $\alpha$ is adjacent to $v$ and that $\deg_H\alpha\geq3$. If $\alpha \not\in C$, then again pick one of edges connecting $v$ and $C$ as $e_v$ and an edge incident to $\alpha$ but not to $v$ as $e_1$. It is clear that the choice can be made such that $o(\ra{e_v})\not=\tilde v$ for any vertex $\tilde v\not=\alpha$ adjacent to $v$. If $\alpha\in C$, then pick for $e_v$ an edge connecting $v$ with $C$ but not incident with $\alpha$ in $H$. Since $\deg_H\alpha\geq 3$ there is an edge in incident with $\alpha$ but not with $o(\ra{e_v})$ and not with $v$. Pick this edge as $e_1$. 
\end{proof}

\begin{figure}[h] 
\begin{tikzpicture}[scale=0.9, every node/.style={transform shape}]
\node (fig1) at (0,0)
       {\includegraphics[scale=0.3]{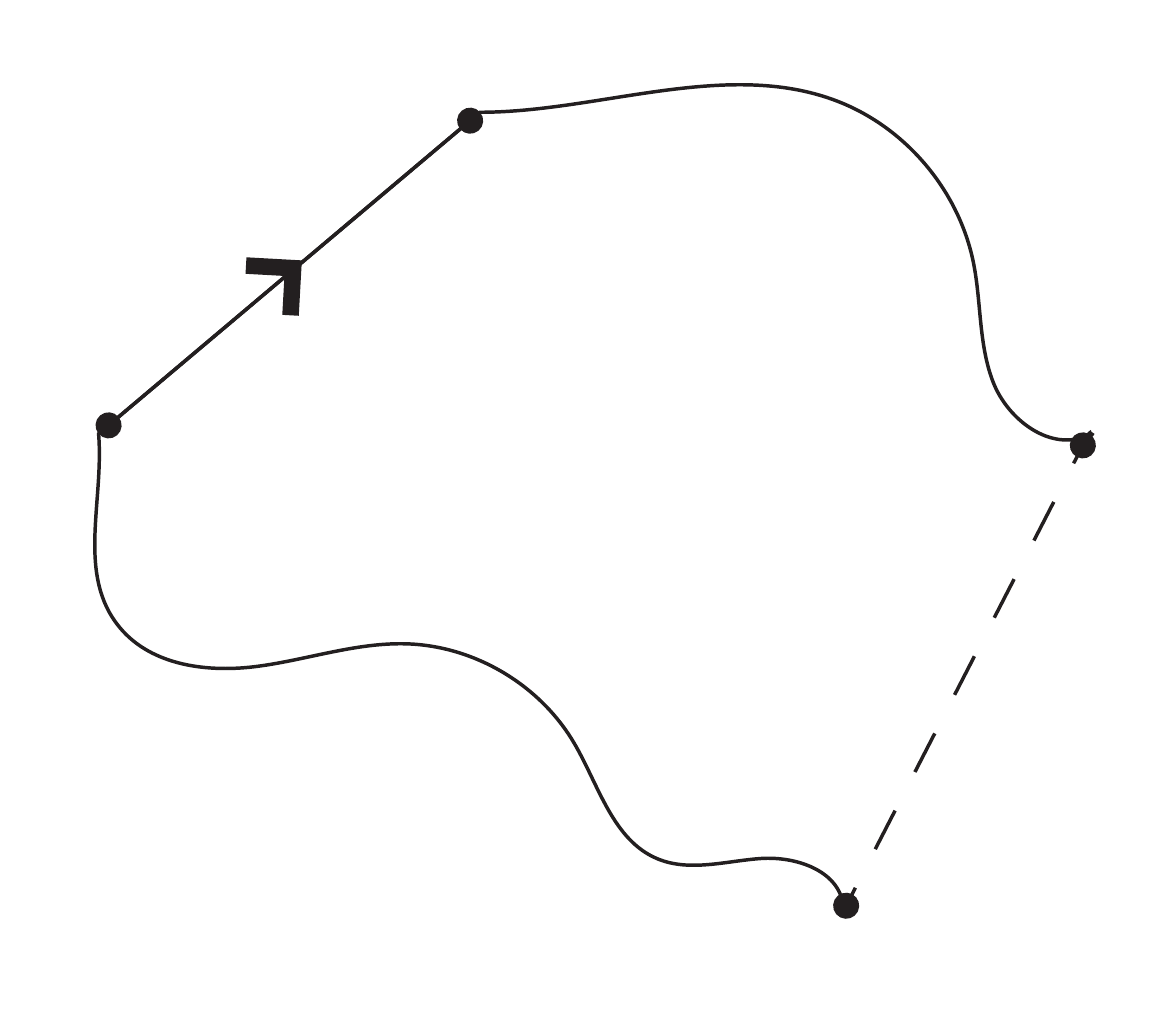}};
 \draw node at (1.4,-0.7) {$e$}; 
  \draw node at (1.7,0.5) {$\omega$}; 
    \draw node at (0.7,-1.5) {$\alpha$}; 
    \draw node at (-1.2,1) {$e_v$};
        \draw node at (-2,0.3) {$o(\ra{e_v})$};
        \draw node at (-0.4,1.5) {$v$};
\end{tikzpicture}
   \caption{A closed walk through $e$ and $e_v$}
    \label{ijk}
\end{figure}
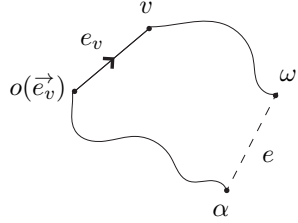

Fix an arbitrary vertex $v \in V-\{\alpha\}$ and assume that either $v$ is not adjacent to $\alpha$ and $\deg_H v \geq 3$, or $\deg_H\alpha\geq3$ and $v$ is adjacent to $\alpha$. Let $\TT$ denote a corresponding tree constructed in the previous lemma (w.r.t.\ a choice of auxiliary vertex $\tilde v$ as indicated).
Now $EH - E\TT$ consists of $b-1$ edges $e_1,\dots,e_{b-1}$, one of which is $e_v$.  Create a basis for the fundamental group of $H$ based at $\alpha$ consisting of closed walks $\gamma_i$ for $i=1,\dots,b-1$ that start in $\alpha$, reach $\ra{e_i}$ via $E\TT$, go through $\ra{e_i}$ and then back to $\alpha$ through $E\TT$. Note that the length of these closed walks can be read off from $H$. Let $\gamma_v$ be the closed walk corresponding to $e_v$. 

If $\gamma$ and $\gamma'$ are two closed walks based at $\alpha$, we denote by $s_v(\gamma,\gamma')=s(\gamma,\gamma')$ the length of the path that two closed walks have in common at the start of the walk.

\begin{theorem} \label{ERC-cond} 
The reconstruction conjecture holds for a graph $G$ with minimal degree $\geq 2$ if we can reconstruct the values $s_v(\gamma_0, \gamma_i)$ for $i=1,\dots,b-1$ and for each vertex $v$ of degree $\geq 3$. 
\end{theorem}

\begin{proof} We start with two lemmas: 

\begin{lemma} \label{length}
If $\gamma$, expressed as a reduced word in the generators $\{\gamma_i\}_{i=0}^{b-1}$, is an arbitrary closed walk in $G$, based at $\alpha$, then the length of $\gamma$ is edge-reconstructible. 
\end{lemma}

\begin{proof} 
We know the lengths of all the generating closed walks $\gamma_i$; for $i>0$, this length can be read off from $H$, and for $i=0$, we know this length by construction. We also know the lengths of the overlaps $s(\gamma_i^{\pm1},\gamma_j^{\pm1})$; for $i>0$ and $j>0$, $s(\gamma_i^{\pm 1}, \gamma_j^{\pm 1})$ can be read off from $H$, and for $i=0$ or $j=0$, we know this length by assumption. 
In a reduced word in $\gamma_i$ and their inverses, no entire closed walk can cancel, since each closed walk contains an edge ($e_i$ for $i>1$ or $e$ for $\gamma_0$) that does not occur in any other generator. Thus, knowing the overlap between generators and their inverses, we know the length of any closed walk given as a word in the generators and their inverses. Concretely:
$$\ell(\gamma_{i_1}^{\pm 1}\dots\gamma_{i_n}^{\pm 1})=\sum_{j=1}^{n}\ell(\gamma_{i_j}^{\pm 1})-2\sum_{j=1}^{n-1} s(\gamma_{i_j}^{\mp 1},\gamma_{i_{j+1}}^{\pm 1}),$$
where the $\gamma_{i_j}^{\pm}$ is the inverse of $\gamma_{i_j}^{\mp}$.
\end{proof} 

\begin{lemma} \label{L}
Let $L[\ra{e_v}]_r$ denote the number of closed walks of length $r$ starting at $\alpha$ that first pass through $\ra{e_v}$ exactly once, and end in $\la{e}$, without passing through $e$ before that. If $\deg v \geq 3$, then the numbers $L[\ra{e_v}]_r$ are edge-reconstructible. 
\end{lemma}

\begin{proof} The number $L[\ra{e_v}]_r$ is the number of words in $\gamma_i^{\pm 1}$ of total length $r$ in which $\gamma_v$ (the generator through $e_v$) occurs once, and ends (in order of composition) with $\gamma_0^{-1}$, which also occurs exactly once. 
\end{proof} 
To reconstruct $L[\ra{e_v}]_r$, one needs to check only the (finitely many) words of word length $\leq r$ in the given generators.
Now let $v$ run through the vertices of $G$ not equal to $\alpha$. Observe that if $\omega$ is adjacent to $\alpha$ in $H$ then $\deg_H\alpha\geq 3$. Define $D[\ra{e_v}]:= \min \{ r \colon L[\ra{e_{v}}]_r \neq 0 \}.$ There are four cases:
\begin{enumerate}
\item If $\deg v \geq 3$, then $\omega=v$ exactly if the minimal $r$ for which $L[\ra{e_v}]_r \neq 0$ equal two more then the distance between $o(\ra{e_v})$ and $\alpha$ in $H-e_v$: 
$$ D[\ra{e_{v}}] = 2+d_{H-e_{v}}(o(\ra{e_{v}}),\alpha)$$  Since we have reconstructed $L[\ra{e_v}]_r$, we have found $\omega$. 
\item If $\deg v = 2$ and both its neighbouring vertices have degree $2$ as well, then $v=\omega$, since we can assume that $G$ has no two adjacent vertices of degree $2$. 
\item If $\deg v = 2$ and $v$ has two neighbouring vertices $v_1$ and $v_2$ of degree $\geq 3$, choose $\ra{e_{v_1}}$ and $\ra{e_{v_2}}$ such that $o(\ra{e_{v_1}})\not\in\{v,\alpha\} $ and $o(\ra{e_{v_2}})\not\in\{v,\alpha\}$. Then $\omega=v$ exactly if $$ D[\ra{e_{v_i}}] = 3+d_{H-e_{v_i}}(o(\ra{e_{v_i}}),\alpha) \mbox{ for $i=1$ and $i=2$} .$$
\item If $\deg v = 2$, and $v$ has exactly one neighbouring vertex $v_1$ of degree two, then $v_1$ has a neighbouring vertex $v_2$ of degree $\geq 3$ and $v$ has a neighbouring vertex $v_3$ of degree $\geq 3$. Again, choose $\ra{e_{v_2}}$ and $\ra{e_{v_3}}$ such that $o(\ra{e_{v_2}})\not\in\{v,\alpha\}$, and $o(\ra{e_{v_3}})\not\in\{v,\alpha\}$. Now $\omega=v$ exactly if 
$$D[\ra{e_{v_i}}] = \left\{ \begin{array}{ll} 4+d_{H-e_{v_2}}(o(\ra{e_{v_2}}),\alpha) \mbox{ if } i=2\\  3+d_{H-e_{v_3}}(o(\ra{e_{v_3}}),\alpha) \mbox{ if } i=3.\end{array} \right.$$

\end{enumerate}
Hence in all cases, we have reconstructed the missing vertex $\omega$. Recall that if $\delta=2$, we had changed the meaning of $H, \alpha$ and $\ell_0$, but after having followed the above procedure to reconstruct $\omega$, in this case, the graph $G$ is found by adding a once-subdivided edge between $\alpha$ and $\omega$ in $H$. 
\end{proof}

\bibliographystyle{amsplain}

\end{document}